\documentclass[12pt,twoside]{article}

\usepackage{amsmath,amssymb,amsthm} 
\usepackage{supertabular}

\newcommand{\Prim}{\mathrm{Prim}}

\newcommand{\Conj}{\mathrm{Conj}}

\newcommand{\vol}{\mathrm{vol}}

\newcommand{\Tr}{\mathrm{Tr}}
\newcommand{\tr}{\mathrm{tr}}
\newcommand{\li}{\mathrm{li}}
\newcommand{\Ind}{\mathrm{Ind}}
\newcommand{\Ker}{\mathrm{Ker}}

\newcommand{\ord}{\mathrm{ord}}
\newcommand{\mr}{\mathrm}

\newcommand{\as}{\quad\text{as}\quad}
\newcommand{\tinf}{\to\infty}
\newcommand{\disp}{\displaystyle}
\newcommand{\bsla}{\backslash}
\newcommand{\nt}{\notag}

\newcommand{\D}{\mathfrak{D}}

\newcommand{\cD}{\mathcal{D}}

\newcommand{\bC}{\mathbb{C}}
\newcommand{\bR}{\mathbb{R}}

\newcommand{\bZ}{\mathbb{Z}}

\newcommand{\bN}{\mathbb{N}}
\newcommand{\noi}{\noindent}

\newcommand{\divset}{\hspace{3pt}|\hspace{3pt}}

\newcommand{\mmid}{\hspace{2pt}||\hspace{2pt}}

\newcommand{\hm}{\hat{m}}

\newcommand{\gam}{\gamma}
\newcommand{\Gam}{\Gamma}
\newcommand{\hGam}{\hat{\Gamma}}

\newcommand{\slt}{\mathrm{SL}_2}
\newcommand{\psl}{\mathrm{PSL}_2}
\newcommand{\sr}{\mathrm{SL}_2(\bR)}
\newcommand{\sz}{\mathrm{SL}_2(\bZ)}

\newcommand{\vapt}{\vspace{3pt}}

\newtheorem{thm}{Theorem}[section]
\newtheorem{prop}[thm]{Proposition}
\newtheorem{lem}[thm]{Lemma}

\makeatletter 
\def\iddots{\mathinner{\mkern1mu\raise\p@
    \hbox{.}\mkern2mu\raise4\p@\hbox{.}\mkern2mu
    \raise7\p@\vbox{\kern7\p@\hbox{.}}\mkern1mu}}
\def\adots{\mathinner{\mkern2mu\raise\p@\hbox{.} 
 \mkern2mu\raise4\p@\hbox{.}\mkern1mu
 \raise7\p@\vbox{\kern7\p@\hbox{.}}\mkern1mu}}
\makeatother

\numberwithin{equation}{section}

\title{Correlations of multiplicities in length spectra for congruence subgroups}
\author{Yasufumi Hashimoto 
\thanks{Department of Mathematical Science, University of the Ryukyus,
Nishihara-cho, Okinawa 903-0213, Japan. e-mail: hashimoto@math.u-ryukyu.ac.jp}}
\date{}

\begin{document}

\abovedisplayskip=2.5pt
\belowdisplayskip=2.5pt

\markboth
{Y. Hashimoto}
{multiplicities in length spectra for congruence subgroups}
\pagestyle{myheadings}

\maketitle
\renewcommand{\thefootnote}{}

\begin{abstract} 
Bogomolny-Leyvraz-Schmit (1996) and Peter (2002)
proposed an asymptotic formula for the correlation of the multiplicities in length spectrum on the modular surface, 
and Lukianov (2007) extended its asymptotic formula to the Riemann surfaces derived from the congruence subgroup $\Gam_0(n)$ 
and the quaternion type co-compact arithmetic groups. 
The coefficients of the leading terms in these asymptotic formulas are described in terms of 
Euler products over prime numbers, and they appear in eigenvalue statistic formulas found by Rudnick (2005) and Lukianov (2007)
for the Laplace-Beltrami operators on the corresponding Riemann surfaces.
In the present paper, we further extend their asymptotic formulas 
to the higher level correlations of the multiplicities 
for any congruence subgroup of the modular group.
\end{abstract}

\section{Introduction}

It is known that there are deep connections between the geometry on hyperbolic manifolds 
and the spectra of the Laplace-Beltrami operators on the corresponding manifolds.
In fact, two compact Riemann surface have the same length spectra if and only if
the Laplacians on these surfaces have the same spectra \cite{Bu,Hu}.
And the Selberg trace formula describes a relation between the lengths of the primitive closed geodesics 
and the eigenvalues of the  Laplace-Beltrami operator \cite{Se}. 
Such a situation is quite similar to the quantum theory; 
Gutzwiller's trace formula describes connections between the periodic orbits in the classical models 
and the energy spectra in the quantum models, 
although its rigorous proof has not been given yet \cite{Gu}. 

The asymptotic formula, called the prime geodesic theorem, of the number of the primitive closed geodesics
on a hyperbolic manifold with the finite volume is given due to Selberg's trace formula \cite{Se,Ga,GW}. 
This is interpreted as an analogue of the prime number theorem 
which is the asymptotic formula for the number of rational prime numbers, 
since the leading terms of both formulas are similar and the error terms are given by 
the non-trivial zeros of the corresponding zeta functions.
However, the detail distributions of the prime numbers and the primitive geodesics are different. 
In fact, there are many geodesics with the same lengths in Riemann surfaces 
and the number of geodesics with the same length (the multiplicity in the length spectrum) is unbounded \cite{Ra}, 
despite there are no such primes. 
The results on numerical experiments imply that 
(the fundamental group of) the manifold is arithmetic if and only if 
the multiplicity is highly increasing, 
and some experts on arithmetic quantum chaos pointed out that 
such a phenomenon is strongly connected to the observations 
that the eigenvalue statistics of the Laplacian 
seems to be the Poisson distribution on arithmetic surfaces, 
and the Gaussian orthogonal ensemble (GOE) 
on non-arithmetic surfaces \cite{BGGS,BLS,LS1,Ma}.
Note that any theoretical proof for such situations has not been given yet 
(see \cite{Take,LS2,Sc,GL}).

The aim in the present paper is to study the growth of the multiplicity in length spectra 
for Riemann surfaces derived from congruence subgroups of the modular group.
Though writing down the length spectra and its multiplicity with elementary objects is not easy in general,  
the multiplicities for the congruence subgroups are written 
in terms of the class numbers of indefinite binary quadratic forms \cite{Sa1,H1}.
Applying results and approaches for the class numbers in the classical analytic number theory, 
Bogomolny-Leyvraz-Schmit \cite{BLS} and Peter \cite{Pe} presented the asymptotic formulas 
for the sum of the shifted product of multiplicities for the modular surface, 
whose coefficient in the leading term is explicitly drawn as an Euler product over prime numbers.
Lukianov \cite{Lu} has obtained similar asymptotic formulas 
for the Riemann surface whose fundamental group is $\Gam_0(n)$, a congruence subgroup of the modular group, 
or a co-compact arithmetic group derived from the indefinite quaternion algebra. 
Rudnick \cite{Ru} and Lukianov \cite{Lu} 
found that the coefficient of the leading term in the asymptotic formulas for the square sum of the multiplicities in length spectra
appear in the formula for statistic behavior of the eigenvalues of the Laplace-Beltrami operators. 

The main result in this paper is to propose asymptotic formulas for the sum of higher shifted product of the multiplicities in length spectrum
for any congruence subgroup of the modular group with an explicit description of the coefficient of their leading term, 
which is an extension of the results in \cite{BLS,Pe,Lu}.  
Since the multiplicities are given by the class numbers of quadratic forms, 
they are approximated by periodic functions \cite{Pe}.
Thus, using the approaches in the theory of arithmetic functions \cite{SS},   
we can get the leading terms of the asymptotic formulas for the products 
of multiplicities and the expressions of the coefficients of the leading terms 
as products over prime numbers. 

The basic notations and the main result in the present paper are as follows.

\noi {\bf Notations and the main result.} 
Let $H:=\{x+y\sqrt{-1},y>0\}$ be the upper half plane and $\Gam$ a discrete subgroup of $\sr$ with $\vol(\Gam\bsla H)<\infty$. 
Denote by $\Prim(\Gam)$ the set of primitive hyperbolic conjugacy classes of $\Gam$ 
and $N(\gam)$ the square of the larger eigenvalue of $\gam\in \Prim(\Gam)$. 
It is known that 
\begin{align}\label{pgt}
\pi_{\Gam}(x):=\#\{\gam\in\Prim(\Gam)\divset N(\gam)<x\}\sim\li(x) \as x\tinf, 
\end{align}
where $\li(x):=\int_2^{x}(\log{t})^{-1}dt$. 
Since there is a one-to-one correspondence between elements in $\Prim(\Gam)$ and primitive closed geodesics on $\Gam\bsla H$
and $N(\gam)$ coincides with the exponential of the length of the corresponding geodesic,
the asymptotic formula \eqref{pgt} is called {\it the prime geodesic theorem} (see, e.g. \cite{Se,He}). 

Let $\Tr(\Gam)$ be the set of $\tr{\gam}$ of $\gam\in\Prim(\Gam)$ and 
$m_{\Gam}(t)$ for $t\in\Tr{\Gam}$ be the number of $\gam\in\Prim(\Gam)$ with $\tr{\gam}=t$.  
Since $\tr{\gam}=N(\gam)^{1/2}+N(\gam)^{-1/2}$ and $N(\gam)^{1/2}=(\tr{\gam}+\sqrt{(\tr{\gam})^2-4})/2$, 
the set $\{(t,m_{\Gam}(t))\}_{t\in\Tr(\Gam)}$ 
is identified to the length spectrum on $\Gam\bsla H$ and the prime geodesic theorem \eqref{pgt} is written by 
\begin{align}\label{pgt1}
\pi_{\Gam}(x^2)=\sum_{\begin{subarray}{c}t\in \Tr(\Gam) \\ t<x\end{subarray}}m_{\Gam}(t)\sim \li(x^ 2).
\end{align}
It is obvious that $\Tr(\sz)=\bZ_{\geq3}$ and $\Tr(\Gam)\subset\bZ_{\geq3}$ for $\Gam\subset \sz$. 
Then the prime geodesic theorem \eqref{pgt1} for such $\Gam$ is given as a sum over integers.
\begin{align}\label{pgt2}
\pi_{\Gam}(x^2)=\sum_{3\leq t<x}m_{\Gam}(t)\sim \li(x^ 2).
\end{align}
For $\Gam=\sz$ and an integer $r\geq0$, 
Bogomolny-Leyvraz-Schmit \cite{BLS} and Peter \cite{Pe} presented the following asymptotic formula. 
\begin{align}\label{peter}
\pi_{\sz}^{(2)}(x^2;r):=\sum_{3\leq t<x}m_{\sz}(t)m_{\sz}(t+r)\sim c_{\sz}^{(2)}(r)\li_2(x^3),
\end{align}
where $\li_2(x):=\int_{2}^x (\log{t})^{-2}dt$ and 
$c_{\sz}^{(2)}(r)$ is a constant described in \cite{BLS,Pe} as a product over prime numbers. 
Lukianov \cite{Lu} also proposed asymptotic formulas for $r=0$ and for $\Gam=\Gam_0(n)$ with square free $n$ 
or a quaternion type co-compact arithmetic $\Gam$. 

The aim in the present paper is to extend the asymptotic formula \eqref{peter} 
to the higher shifted product of the multiplicities for any congruence subgroup.
For an integer $n\geq1$, let $\bZ_{n}:=\bZ/n\bZ$ and 
\begin{align*}
&\Gam(n):=\Ker\big(\mr{SL}_2(\bZ)\stackrel{\text{proj.}}{\rightarrow}
\mr{SL}_2(\bZ_n)\big)
=\{\gam\in \mr{SL}_2(\bZ)\divset \gam\equiv I\bmod{n} \},\\
&\hat{\Gam}(n):=\Ker\big(\mr{SL}_2(\bZ)\stackrel{\text{proj.}}{\rightarrow}
\mr{PSL}_2(\bZ_n)\big)
=\{\gam\in \mr{SL}_2(\bZ)\divset \gam\equiv \alpha I\bmod{n}, \alpha^2\equiv 1\bmod{n} \},
\end{align*}
the principal congruence subgroups of level $n$. 
Throughout in this paper, we call $\Gam$ a congruence subgroup of level $n$ 
if $\hGam(n)\subset\Gam\subset\sz$ and $\Gam\not\supset \hGam(m)$ for any $m<n$. 

The main result in this paper is to extend \eqref{peter} as follows.
\begin{thm}\label{thm0}
Let $k\geq2$ be an integer, ${\bf r}:=(r_1,\cdots,r_k)\in\bZ^{k}$ and
$\Gam$ a congruence subgroup of $\sz$. 
Then we have 
\begin{align*}
\pi_{\Gam}^{(k)}(x^2;{\bf r}):=\sum_{3\leq t<x}m_{\Gam}(t+r_1)\cdots m_{\Gam}(t+r_k)
\sim c_{\Gam}^{(k)}({\bf r})\li_k(x^{k+1}),
\end{align*}
where $\li_k(x):=\int_{2}^x (\log{t})^{-k}dt$ and $c_{\Gam}^{(k)}({\bf r})$ is a constant described in Theorem \ref{thm}.
\end{thm}

\section{Length spectra for congruence subgroups}

In this section, we propose an expression of the multiplicities in length spectrum for a congruence subgroup 
in terms of indefinite binary quadratic forms. 

\subsection{Quadratic forms and the modular group}
Let $Q(x,y)=[a,b,c]:=ax^2+bxy+cy^2$
be a binary quadratic form over $\bZ$ with $a,b,c\in\bZ$ and $\gcd(a,b,c)=1$. 
Denote by $D=D(Q):=b^2-4ac$ the discriminant of $[a,b,c]$. 
We call quadratic forms $Q$ and $Q'$ equivalent 
and write $Q\sim Q'$ if there exists $g\in\sz$ such that $Q(x,y)=Q'\big((x,y).g\big)$. 
Denote by $h(D)$ the number of equivalence classes of the quadratic forms of given $D=b^2-4ac$. 
It is known that, if $D>0$, then there are infinitely many positive solutions $(t,u)$ 
of the Pell equation $t^2-Du^2=4$. 
Put $(t_j,u_j)=(t_j(D),u_j(D))$ the $j$-th positive solution of $t^2-Du^2=4$ 
and $\epsilon_j(D):=(t_j(D)+u_j(D)\sqrt{D})/2$. 
Note that $\epsilon(D)=\epsilon_1(D)$ is called the fundamental unit of $D$ in the narrow sense, 
and it holds that $\epsilon_j(D)=\epsilon(D)^j$.

For a quadratic form $Q=[a,b,c]$ and a positive solution $(t,u)$ of $t^2-Du^2=4$, let 
\begin{align}\label{1to1}
\gam\big(Q,(t,u)\big):=\begin{pmatrix}\disp\frac{t+bu}{2}&-cu\\ au& \disp\frac{t-bu}{2}\end{pmatrix}\in\slt(\bZ).
\end{align}
Conversely, for $\gam=(\gam_{ij})_{1\leq i,j\leq 2}\in \sz$, we put
\begin{align}
&t_{\gam}:=\gam_{11}+\gam_{22},\quad u_{\gam}:=\gcd{(\gam_{21},\gam_{11}-\gam_{22},-\gam_{12})},\notag\\
&a_{\gam}:=\gam_{21}/u_{\gam},\quad b_{\gam}:=(\gam_{11}-\gam_{22})/u_{\gam},
\quad c_{\gam}:=-\gam_{12}/u_{\gam}, \label{ttod}\\
&Q_{\gam}:=[a_{\gam},b_{\gam},c_{\gam}],\quad 
D_{\gam}:=\frac{t_{\gam}^2-4}{u_{\gam}^2}=b_{\gam}^2-4a_{\gam}c_{\gam}.\notag
\end{align}
It is known that \eqref{1to1} and \eqref{ttod} give a one-to-one correspondence between
equivalence classes of primitive binary quadratic forms with $D>0$ 
and primitive hyperbolic conjugacy classes of $\sz$ 
(see \cite{Sa1} and Chap. 5 in \cite{G}). 
Then the multiplicity $m_{\Gam}(t)$ for $\Gam=\sz$ is described as follows (see, e.g. \cite{Sa1}).
\begin{prop}\label{hd} 
Let $t\geq 3$ be an integer. Then we have
\begin{align*}
m_{\sz}(t)=\sum_{u\in U(t), j_{t,u}=1}h\left(D_{t,u} \right),
\end{align*}
where $D_{t,u}:=(t^2-4)/u^2$, $U(t):=\{u\geq1 \divset u^2\mid t^2-4, D_{t,u}\in \D\}$ and 
\begin{align*}
j_{t,u}:=\max\left\{j\geq1\divset \epsilon_j\left( D_{t,u}\right)=\frac{1}{2}\left(t+\sqrt{t^2-4}\right)\right\}.
\end{align*}
\end{prop}

\subsection{Conjugations in $\psl(\bZ/n\bZ)$}

For an integer $n\geq1$, let $\bZ_{n}^*$ be the multiplicative group in $\bZ_n:=\bZ/n\bZ$ and  
$\bZ_{n}^{(2)}:=\bZ_{n}^*/\left(\bZ_{n}^*\right)^2$. 
Note that 
\begin{align*}
\bZ_{p^r}^{(2)}=\begin{cases} \{1,3,5,7\},&(p=2,r\geq3),\\
\{1,\eta\},&(p\geq3,r\geq1), \end{cases}
\end{align*}
where $\eta$ is a non quadratic residue of $p$.

\begin{lem}\label{conj01}
Let $n\geq1,a,b,c$ be integers with $\gcd(a,b,c,n)=1$. 
Denote by $D:=b^2-4ac$ and suppose that $t,u\geq1$ satisfies $t^2-Du^2=4$.
Put  
\begin{align*}
\gam:=\begin{pmatrix}\disp\frac{1}{2}(t+bu)&-cu\\au&\disp\frac{1}{2}(t-bu)\end{pmatrix}, \quad
\gam_{\nu}=\gam_{\nu}(D):=\begin{pmatrix}\disp\frac{t+\delta u}{2}&\disp\frac{D-\delta^2}{4}\nu^{-1}u\\ 
\nu u&\disp\frac{t-\delta u}{2}\end{pmatrix},
\end{align*}
where $\nu\in\bZ_n^{*}$, and $\delta$ is $1$ when $D\equiv 1\bmod{4}$ and is $0$ otherwise. 
Then, for any $\gam \in \sz$, there exists $\nu\in \bZ_n^{(2)}$ such that $\gam\sim\gam_{\nu}$ in $\mr{PSL}_2(\bZ_n)$.
\end{lem}

\begin{proof}
When $\gcd(a,n)\neq1$, it is easy to see that there exists $g'\in\mr{PSL}_2(\bZ_n)$ such that 
$\gcd\left(({g'}^{-1}\gam g')_{21}/u,n\right)=1$. 
Then the problem is reduced to the case of $\gcd(a,n)=1$. 
When $\gcd(a,n)=1$, put $\alpha\in\bZ_n^*$ and $\nu\in \bZ_n^{(2)}$ such that $a=\nu\alpha^2$.  
We have ${g}^{-1}\gam g=\gam_{\nu}$ where 
\begin{align*}
g:=\begin{pmatrix}\alpha^{-1}&\disp(\nu\alpha)^{-1}(b+\delta)/2\\
0&\alpha\end{pmatrix}.
\end{align*}
Thus the claim in Lemma \ref{conj01} holds.
\end{proof}

\begin{lem}\label{conj2}
Let $p$ be a prime and $r\geq1$ an integer. 
Then, for any $\nu\in\bZ_{p^r}^{(2)}$, 
there exists $l_1,l_2\geq1$ such that $\gam_{1}^{l_1}\sim\gam_{\nu}$ and $\gam_1\sim\gam_{\nu}^{l_2}$ in $\psl(\bZ_{p^r})$.
\end{lem}

\begin{proof} 

\noi{\bf The case of $p=2$.} 
If $2\nmid t$ then $D\equiv 5\bmod{8}$ and $\delta=1$ since $Du^2=t^2-4$.
In this case, we see that 
\begin{align*}
(g_2^{-1}\gam_1 g_2)_{21}/u\equiv 7,\quad (g_4^{-1}\gam_1 g_4)_{21}/u\equiv 3, \quad (g_6^{-1}\gam_1 g_6)_{21}/u\equiv 5\bmod{8},
\end{align*}
where $g_i:=\begin{pmatrix}1& 0\\ i & 1\end{pmatrix}$. 
Thus, according to Lemma \ref{conj01}, we have
\begin{align}
2\nmid t \quad  \Rightarrow \quad \gam_1\sim\gam_3\sim\gam_5\sim \gam_7 \quad \text{in $\psl(\bZ_{2^{r}})$.}
\end{align}

If $4\mid t$ then $D\equiv -4\bmod{16}$ and $\delta=0$. 
In this case, we see that 
\begin{align*}
(g_2^{-1}\gam_1 g_2)_{21}/u\equiv 5,\quad (g_2^{-1}\gam_3 g_2)_{21}/u\equiv 7\bmod{8},
\end{align*}
namely 
\begin{align}
4\mid t \quad  \Rightarrow \quad \gam_1\sim\gam_5,\quad \gam_3\sim \gam_7.
\end{align}
Since $(t_j+u_j\sqrt{D})/2=((t+u\sqrt{D})/2)^j$, we have $u_3/u=t^2-1$. 
This gives that $u_3/u_1\equiv -1\bmod{8}$ for $4\mid t$, 
and then we get 
\begin{align}
4\mid t \quad  \Rightarrow \quad 
\gam_1^3\sim \gam_7,\quad \gam_3^3\sim\gam_5,\quad \gam_5^3\sim\gam_3,\quad \gam_7^3\sim\gam_1. 
\end{align}

We can similarly obtain 
\begin{align}
t\equiv 2\mod{4} \quad \Rightarrow \quad 
&\gam_1^{\nu}\sim \gam_{\nu},\quad \gam_{\nu}^{\nu}\sim \gam_1 \quad \text{for any $\nu\in\bZ_{p^r}^{(2)}$.}
\end{align}

\noi {\bf The case of $p= 3$.}
We have $(g^{-1}\gam_1 g)_{21}\equiv 2u$ mod $3^r$, 
where 
\begin{align*}
g:=\begin{cases}\disp\begin{pmatrix}0&-(-D/8)^{-1/2}\\(-8/D)^{1/2}&0\end{pmatrix},
&D\equiv1\bmod{3},\\  
\disp\begin{pmatrix}(2+D/4)^{1/2}&(2+D/4)^{1/2}-1\\ 1&1 \end{pmatrix},&D\equiv2\bmod{3}. 
\end{cases}
\end{align*}
Then $\gam_1\sim\gam_2$ in $\mr{PSL}_2(\bZ_{3^r})$ for $3\nmid D$. 
For $3\mid D$, it is easy to see that $\gam_1^2\sim \gam_2$ and $\gam_2^2\sim \gam_1$.

\vapt

\noi {\bf The case of $p\geq 5$.} 
First, study the existence of solutions $(x,y)$ of the equation $x^2-dy^2\equiv \eta\bmod{p^r}$ where $(\eta/p)=-1$.
When $p\mid d$, there are no solutions.
When $(d/p)=1$, $x=(\eta+1)/2$ and $y=(\eta-1)/(2d^{1/2})^{-1}$ is a solution. 
When $(d/p)=-1$, let $d_1\equiv (d\eta^{-1})^{1/2}$. 
Since the equation is written by $x^2\equiv \eta(1+d_1^2y^2)$, 
choosing $y$ such that $1+d_1^2y^2$ is a non quadratic residue, 
we see that the equation has a solution.

Thus, if $p\nmid d$, we see that $(g^{-1}\gam_1g)_{21}\equiv \eta$ and then $\gam_1\sim\gam_{\eta}$, 
where $g$ is given such that $(x,y)=(g_{11},g_{21})$ is a solution of $x^2-dy^2\equiv \eta\bmod{p^r}$.
If $p\mid D$, similar to the case of $p=2$, we have 
\begin{align*}
\frac{u_{\eta}}{u_1}=\sum_{l=0}^{(\eta-1)/2} \binom{\eta}{2l+1}\left(\frac{t}{2}\right)^{\eta-2l-1}(t^2-4)^l,
\end{align*}
where $\eta$ is chosen to be odd.
Since $p\mid D\mid t^2-4$, we get
\begin{align*}
\frac{u_{\eta}}{u_1}\equiv \eta\left((t/2)^{(\eta-1)/2}\right)^2\equiv \eta\alpha^2
\end{align*}
for some $\alpha\in\bZ_{p^r}^*$.
Thus it holds that $\gam_1^\eta\sim \gam_{\eta}$ and $\gam_{\eta}^{\nu}\sim \gam_1$ where $\eta\nu\equiv 1\bmod{p}$.
\end{proof}

\subsection{Length spectra for congruence subgroups}

According to Venkov-Zograf's formula \cite{VZ} for Selberg's zeta function, we have
\begin{align}\label{venkov}
\hm_{\Gam}(t):=\sum_{\begin{subarray}{c}\gam\in\Prim(\Gam),j\geq1\\ t_{\gam^j}=t \end{subarray}}\frac{1}{j}
=\sum_{\begin{subarray}{c}\gam\in\Prim(\sz),j\geq1\\ t_{\gam^j}=t \end{subarray}}\frac{1}{j}\tr{\chi_{\Gam}(\gam^j)},
\end{align}
where $\chi_{\Gam}:=\Ind_{\Gam}^{\sz}1$. 
The value $\hm_{\Gam}(t)$ is expressed as follows.

\begin{prop}\label{arithcong}
Let $n\geq1$, $t\geq3$ be integers and $\Gam$ a congruence subgroup of level $n$.
Then we have
\begin{align}\label{arith}
\hm_{\Gam}(t)=\sum_{u\in U(t)}\frac{1}{j_{t,u}}\omega_{\Gam}(t,u)h\left( D_{t,u} \right),
\end{align}
where $\omega_{\Gam}(t,u):=\tr\chi_{\Gam}\left(\gam_1\left( D_{t,u}\right)  \right)$ and 
$\gam_1\left( D_{t,u}\right)$ is given in Lemma \ref{conj01}.
\end{prop}

\begin{proof}
First study the case of $n=p^r$.
Due to Lemma \ref{conj01} and \ref{conj2}, we see that, for any $\gam,\gam'\in \sz$ with $D_{\gam}=D_{\gam'}$, 
there exists $\eta,\nu\in \bZ$ such that $\gam^{\nu}\sim\gam'$ and $\gam'^{\eta}\sim \gam$ in $\psl(\bZ_{p^r})$.
Since $\chi_{\Gam}=\Ind_{\Gam}^{\sz}1\simeq \Ind_{\Gam/\hGam(p^r)}^{\psl(\bZ_{p^r})}1$ 
is a permutation representation of $\psl(\bZ_{p^r})$, 
we see that $\chi_{\Gam}(\gam)\sim\chi_{\Gam}(\gam')$, 
namely $\tr\chi_{\Gam}(\gam)$ depends only on $D_{\gam}$. 
Thus \eqref{arith} for $n=p^r$ follows from \eqref{venkov}.

When $n$ is factored by $n=\prod_{p\mid n}p^r$, it holds that $\chi_{\Gam}(\gam)=\bigotimes_{p\mid n}\chi_{\Gam\hGam(p^r)}(\gam)$ (see \cite{HW}). 
Thus \eqref{arith} holds also for any $n$.
\end{proof}

\section{Proof of Theorem \ref{thm0}}

\subsection{Description of $c_{\Gam}^{(k)}({\bf r})$}

The coefficient $c_{\Gam}^{(k)}({\bf r})$ in Theorem \ref{thm0} is described in the following theorem.

\begin{thm}\label{thm}
Let $t\geq3$ be an integer, $\Gam$ a congruence subgroup of $\sz$ and 
\begin{align*}
I_{\Gam}(t):=\frac{\log{\left(\frac{1}{2}(t+\sqrt{t^2-4})\right)}}{\sqrt{t^2-4}}\hm_{\Gam}(t).
\end{align*}
Then, for any $k\geq1$ and ${\bf r}:=(r_1,\cdots,r_k)\in \bZ^{k-1}$, the limit 
\begin{align*}
c_{\Gam}^{(k)}({\bf r}):=&\lim_{x\tinf}\frac{1}{x}\sum_{3\leq t\leq x}I_{\Gam}(t+r_1)\cdots I_{\Gam}(t+r_k)
\end{align*}
exists and coincides with
\begin{align*}
c_{\Gam}^{(k)}({\bf r})=&\prod_{p}\Bigg(\lim_{l\tinf} p^{l(k-1)}\sum_{m\in \bZ_{p^l}}F_{\Gam}(m+r_1;p^l)\cdots F_{\Gam}(m+r_k;p^l)\Bigg),
\end{align*}
where 
\begin{align*}
F_{\Gam}(m;n):=\frac{\#\{\gam\in\Gam(n)\bsla \Gam(n)\Gam\divset \tr{\gam}\equiv m\bmod{n} \}}{\#\Gam(n)\bsla \Gam(n)\Gam}.
\end{align*}
\end{thm}

We now prove that Theorem \ref{thm0} follows from Theorem \ref{thm} with the same $c_{\Gam}^{(k)}({\bf r})$. 

\vapt

\noi {\it Proof of ``Theorem \ref{thm} $\Rightarrow$ Theorem \ref{thm0}".} 
Let 
\begin{align*}
\hat{\pi}_{\Gam}^{(k)}(x;{\bf r}):=&\sum_{3\leq t\leq x}\hat{m}_{\Gam}(t+r_1)\cdots\hat{m}_{\Gam}(t+r_k).
\end{align*}
It is easy to see that if Theorem \ref{thm} holds then 
\begin{align*}
\hat{\pi}_{\Gam}^{(k)}(x;{\bf r})\sim c_{\Gam}^{(k)}({\bf r})\li_k(x^{k+1}).
\end{align*}

We now compare $\pi_{\Gam}^{(k)}(x;{\bf r})$ and $\hat{\pi}_{\Gam}^{(k)}(x;{\bf r})$ as follows. 
\begin{align*}
\left|\hat{\pi}_{\Gam}^{(k)}(x;{\bf r})-\pi_{\Gam}^{(k)}(x;{\bf r})\right|
\leq &\sum_{1\leq i\leq k}\sum_{3\leq t\leq x}\hat{m}_{\Gam}(t+r_1)\cdots\hat{m}_{\Gam}(t+r_{i-1})
\\&\times |\hat{m}_{\Gam}(t+r_i)-m_{\Gam}(t+r_i)| m_{\Gam}(t+r_{i+1})\cdots m_{\Gam}(t+r_k).
\end{align*}
The classical bound of the class numbers $h(D_{t,u})\ll D_{t,u}^{1/2+\epsilon}\ll t^{1+\epsilon}$ 
gives the estimates $\hm_{\Gam}(t),m_{\Gam}(t)\ll t^{1+\epsilon}$.
Thus we get 
\begin{align}
\left|\hat{\pi}_{\Gam}^{(k)}(x;{\bf r})-\pi_{\Gam}^{(k)}(x;{\bf r})\right|
\ll \sum_{1\leq i\leq k}\sum_{\begin{subarray}{c}\gam\in \Prim(\Gam),j\geq2\\ t_{\gam^j}<x\end{subarray}}
t^{k-1+\epsilon}\frac{1}{j}\ll x^{k+1/2+\epsilon}.\label{hat}
\end{align} 
The claim follows immediately.
\qed

\subsection{Approximation of $I_{\Gam}(t)$ by periodic functions}

In this subsection, we study $I_{\Gam}$ in the view of the theory on arithmetic functions \cite{SS}.

For an integer $q\geq1$ and a function $f:\bN\to\bC$, define the semi-norm
\begin{align*}
||f||_q:=\left( \limsup_{x\tinf}\frac{1}{x} \sum_{1\leq n\leq x}|f(n)|^q \right)^{1/q}.
\end{align*}
The function $f$ is called a $q$-limit periodic function 
if, for any $\epsilon>0$, there is a periodic function $h$ such that $||f-h||_q<\epsilon$.
The set $\cD^q$ of all $q$-limit periodic functions becomes a Banach space
with the norm $||*||_q$ if functions $f_1,f_2$ with $||f_1-f_2||_q=0$ are identified. 

We now prove the following proposition.

\begin{prop}\label{lemproduct2}
Let $q\geq1$ be an integer and $f_1,\cdots,f_q\in\cD^q$. 
Suppose that $f_i$ is approximated by a series of periodic functions $\{f_{ij}\}_{j\geq1}$, 
namely $||f_{ij}-f_i||_q\to 0$ as $j\tinf$.
Without loss of generality, suppose that 
$f_{1j},\cdots,f_{qj}$ have the same period $N_j$ and $N_j\tinf$ as $j\tinf$.
Then we have 
\begin{align*}
\lim_{x\tinf}\frac{1}{x}\sum_{1\leq t\leq x}f_1(t)\cdots f_q(t)
=\lim_{j\tinf}N_j^{q-1}\sum_{m\in\bZ_{N_j}}F_{1j}(m;N_j)\cdots F_{qj}(m;N_j),
\end{align*}
where 
\begin{align*}
F_{ij}(m;N_j):=\lim_{x\tinf}\frac{1}{x}\sum_{\begin{subarray}{c} 1\leq t\leq x \\ t\equiv m \bmod{N_j}\end{subarray}}f_{ij}(t)
=\frac{1}{N_j}f_{ij}(m).
\end{align*}
\end{prop}

\begin{proof}
Let $q,j\geq1$ be integers, $x>0$ a number sufficiently larger than $N_j$ and 
\begin{align*}
G(x):= \frac{1}{x}\sum_{1\leq t\leq x}f_1(t)\cdots f_q(t),\qquad
G_j(x):=\frac{1}{x}\sum_{1\leq t\leq x}f_{1j}(t)\cdots f_{qj}(t).
\end{align*}
Since $f_{1j},\cdots,f_{qj}$ are $N_j$-periodic, we have
\begin{align}
G_j(x)=&\frac{1}{x}\sum_{0\leq l\leq x/N_j}\sum_{0\leq m< N_j}f_{1j}(N_jl+m)\cdots f_{qj}(N_jl+m)
+\frac{1}{x}\sum_{N_j[x/N_j]\leq t\leq x}f_{1j}(t)\cdots f_{qj}(t)\notag\\
\to&N_j^{q-1}\sum_{0\leq m< N_j}F_{1j}(m;N_j)\cdots F_{qj}(m;N_j) \as x\tinf.\label{periodic}
\end{align}

The assumption $||f_{i}-f_{ij}||_q\to 0$ as $j\tinf$ gives that 
\begin{align}
&\limsup_{x\tinf}\left|G(x)-G_j(x)\right| 
\leq ||f_1\cdots f_q-f_{1j}\cdots f_{qj}||_1 \nt\\
\leq & \sum_{1\leq l\leq q}||f_1||_q\cdots ||f_{l-1}||_q ||f_l-f_{lj}||_q ||f_{l+1,j}||_q\cdots||f_{qj}||_q\to 0 \as j\tinf.\label{limsup}
\end{align}
Since $\lim_{x\tinf}G_j(x)$ exists for any $j$ and $G(x)$ does not depend on $j$, 
the clam in this proposition follows from \eqref{periodic} and \eqref{limsup}. 
\end{proof}

Next, study the $q$-limit periodicity of $I_{\Gam}$.
By the definition of $I_{\Gam}$ and the class number formula, we have
\begin{align*}
I_{\Gam}(t)=\sum_{u\in U(t)}\omega_{\Gam}(t,u)u^{-1}L(1,D_{t,u}),
\end{align*}
where $L(1,D_{t,u}):=\prod_p\left( 1-(D_{t,u}/p)p^{-1} \right)^{-1}$. 
For integers $P\geq2$ and $M\geq1$, let 
\begin{align*}
\beta_{\Gam,P,M}(t):=&\sum_{\begin{subarray}{c}u\in U(t)\\ p\mid u\Rightarrow p\leq P\\ \ord_p u\leq M\end{subarray}}\omega_{\Gam}(t,u)u^{-1}
\prod_{p\leq P}\left( 1-(D_{t,u}/p)p^{-1} \right)^{-1},
\end{align*}
and $\beta_{\Gam,P}(t):=\lim_{M\tinf}\beta_{\Gam,P,M}(t)$.
For $\Gam=\sz$, we see that $\beta_{\Gam,P,M}(t)$ is of $B_{P,M}:=2^{2M+3}\prod_{2<p\leq P}p^{2M+1}$-periodic. 
Peter \cite{Pe} approximated $I_{\sz}$ by $\beta_{\sz,P}$, $\beta_{\sz,P,M}$, 
and proved that $I_{\sz}\in\cD^q$ for any $q\geq1$. 

For a congruence subgroup $\Gam$ of level $n$, 
we see that $\beta_{\Gam,P,M}(t)$ is of $n^2B_{P,M}$-periodic 
since $\omega_{\Gam}(t,u)$ is a character on $\psl(\bZ_n)$.
We prove in the following lemma, due to Peter's work \cite{Pe}, 
that $I_{\Gam}\in \cD^q$ for any congruence subgroup $\Gam$ and any $q\geq1$

\begin{lem}\label{Idq}
For any $q\geq1$, there exists a constant $\epsilon>0$ such that 
\begin{align*}
||I_{\Gam}-\beta_{\Gam,P,M}||_q\ll P^{-\epsilon}+2^{-M}(\log{P})^2\as P,M\tinf,
\end{align*}
where the implied constant depends on $q$ and $\Gam$.
\end{lem}

\begin{proof}
By the definition of $\beta_{\Gam,P,M}(t)$ and $\beta_{\Gam,P}(t)$, we have
\begin{align*}
\left|\beta_{\Gam,P}(t)-\beta_{\Gam,P,M}(t) \right|
=&\sum_{\begin{subarray}{c}u\in U(t)\\ p\mid u\Rightarrow p\leq P\\ \ord_p u\geq M(\exists p)\end{subarray}}\omega_{\Gam}(t,u)u^{-1}
\prod_{p\leq P}\left(1- (D_{t,u}/p)p^{-1}\right)^{-1}.
\end{align*}
Since $0\leq \omega_{\Gam}(t,u)\leq [\sz;\Gam]$, the difference above is bounded by 
\begin{align}
\left|\beta_{\Gam,P}(t)-\beta_{\Gam,P,M}(t) \right|
\leq & [\sz:\Gam] \sum_{2\leq p\leq P} \sum_{\begin{subarray}{c}u\in U(t), p^M\mid u \end{subarray}}u^{-1}
\prod_{p_1\leq P}\left(1- (D_{t,u}/p_1)p_1^{-1}\right)^{-1}\notag 
\end{align}
\begin{align}
\ll & (\log{P})^2 \sum_{2\leq p\leq P}p^{-M} \ll 2^{-M}(\log{P})^2.\label{bpm}
\end{align}
By virtue of Lemma 3.1 and Proposition 3.7 in \cite{Pe} (see also Corollary 4.2) and \eqref{bpm}, we obtain
\begin{align*}
||I_{\Gam}-\beta_{\Gam,P,M}||_q
\leq &||I_{\Gam}-\beta_{\Gam,P}||_q+||\beta_{\Gam,P}-\beta_{\Gam,P,M}||_q\\
\leq &[\sz:\Gam]||I_{\sz}-\beta_{\sz,P}||_q+\sup_{t\geq3}|\beta_{\Gam,P}(t)-\beta_{\Gam,P,M}(t)|_q\\
\ll  &P^{-\epsilon}+2^{-M}(\log{P})^2,
\end{align*}
for some $\epsilon>0$.
\end{proof}

\subsection{Partial sum of multiplicities}

In this subsection, we study the growth of partial sums of $I_{\Gam}(t)$ 
to give the expression of $c_{\Gam}^{(k)}({\bf r})$ in Theorem \ref{thm}. 
First we prepare the following variation of the prime geodesic theorem 
called the Tchebotarev type prime geodesic theorem.

\begin{thm}\label{chebo} (Tchebotarev type prime geodesic theorem, \cite{Sa1} and \cite{Su}) 
Let $\Gam_1,\Gam_2$ be discrete subgroups of $\sr$ with $\vol(\Gam_1\bsla H)<\infty$, $\Gam_1\vartriangleright\Gam_2$, $[\Gam_1:\Gam_2]<\infty$. 
Then, for $[g]\in\Conj(\Gam_2\bsla \Gam_1)$, we have  
\begin{align}\label{chebo1}
\#\{\gam\in\Prim(\Gam_1)\divset \sigma(\gam)\subset [g],N(\gam)<x\}\sim \frac{\#[g]}{\#\Gam_2\bsla \Gam_1}\li(x),
\end{align}
where $\sigma:\Gam_1\to \Gam_2\bsla \Gam_1$ is a projection. \qed
\end{thm}

Due to Theorem \ref{chebo1}, we get the following lemma.

\begin{lem} \label{partial}
Let $N\geq1$ be an integer, $m\in\bZ_N$ and $\Gam$ a congruence subgroup of $\sz$. 
Then we have
\begin{align*}
\lim_{x\tinf}\frac{1}{x}\sum_{\begin{subarray}{c}3\leq t\leq x \\ t\equiv m\bmod{N}\end{subarray}}I_{\Gam}(t)
=\frac{\#\{\gam\in \Gam(N)\bsla \Gam\Gam(N) \divset \tr{\gam}\equiv m\bmod{N}\}}{\#\Gam(N)\bsla \Gam\Gam(N)}.
\end{align*}
\end{lem}
 
\begin{proof}
It is easy to see that $\Gam_1=\Gam$ and $\Gam_2=\Gam\cap\hGam(N)$ satisfy the condition in Theorem \ref{chebo}. 
Since $\tr{\gam_1}\equiv \alpha\tr{\gam_2}\bmod{N}$ for some $\alpha^2\equiv 1\bmod{N}$ if $\gam_1\sim\gam_2$ in $\Gam\cap\hGam(N)\bsla \Gam$, 
we have 
\begin{align}
\lim_{x\tinf}\frac{1}{x}\sum_{\begin{subarray}{c}3\leq t\leq x \\ t\equiv \alpha m\bmod{N}\\ \alpha^2\equiv1\bmod{N}\end{subarray}}I_{\Gam}(t)
=\frac{\#\{\gam\in \Gam\cap\hGam(N)\bsla \Gam\divset \tr{\gam}\equiv \alpha m\bmod{N},\alpha^2\equiv1\bmod{N} \}}
{\#\Gam\cap\hGam(N)\bsla \Gam}.\label{partial1}
\end{align}
Due to the isomorphism theorem, we have 
\begin{align*}
\Gam\cap\hGam(N)\bsla \Gam \simeq \hGam(N)\bsla \Gam\hGam(N)
\end{align*}
with a one-to-one correspondence $\big(\Gam\cap\hGam(N)  \big)\gam \longleftrightarrow \hGam(N)\gam$ $(\gam\in \Gam)$.
Thus the equation \eqref{partial1} is written by 
\begin{align}
\lim_{x\tinf}\frac{1}{x}\sum_{\begin{subarray}{c}3\leq t\leq x \\ t\equiv \alpha m\bmod{N}\\ \alpha^2\equiv1\bmod{N}\end{subarray}}I_{\Gam}(t)
=\frac{\#\{\gam\in \hGam(N)\bsla \Gam\hGam(N) \divset \tr{\gam}\equiv \alpha m\bmod{N},\alpha^2\equiv1\bmod{N} \}}
{\#\hGam(N)\bsla \Gam\hGam(N)}.\label{partial11}
\end{align}

According to Lemma 2.11 and 2.19 in \cite{Rau}, we see that
\begin{align*}
A(m;N;u):=\lim_{x\tinf}\frac{1}{x}\sum_{\begin{subarray}{c} t<x,t\equiv m\bmod{N}\\ D_{t,u}\in\cD \end{subarray}}u^{-1}L(1,D_{t,u})
\end{align*}
satisfies that $A(m;N,u)=A(\alpha m;N,u)$ for any $u\geq1,m\in \bZ_N, \alpha^2\equiv 1\bmod{N}$.
Furthermore, since $\omega_{\Gam}(t,u)$ is a character on ${\rm PSL}_2(\bZ_n)$, taking the sum of $A(m;N;u)$ over $u$, 
we get the equation in Lemma \ref{partial} from \eqref{partial11}.
 \end{proof}

The following lemma describes the values in Proposition \ref{partial} for $\Gam\Gam(p^r)=\sz$. 
\begin{lem}\label{trace}
When $p=2$ and $r\geq6$, we have
\begin{align*}
&\#\left\{\gam\in \mathrm{SL}_2(\bZ_{2^r})\divset \tr{\gam}\equiv t\bmod{2^r}\right\}\\
&=\begin{cases}
2^{2r-1}&(2\nmid t),\\
3\cdot 2^{2r-2}, & (4\mid t),\\
3\cdot 2^{2r-1}, & (t\equiv 16t_1\pm2\bmod{2^r},t_1\equiv5\bmod{8},\\
 & \text{ or $t\equiv 2^lt_1\pm2\bmod{2^r},t_1\equiv1\bmod{8}$, $l\geq6$: even}),\\
5\cdot 2^{2r-2}, & (t\equiv 16t_1\pm2\bmod{2^r},t_1\not\equiv5\bmod{8}),\\
3\cdot 2^{2r-1}-2^{2r-l/2}, & (\text{$t\equiv 2^lt_1\pm2\bmod{2^r},t_1\not\equiv1\bmod{8}$, $l\geq6$: even}),\\
3\cdot (2^{2r-1}-2^{2r-(l+3)/2}), & (\text{$t\equiv 2^lt_1\pm2\bmod{2^r},2\nmid t_1$, $l\geq3$: odd}),\\
3\cdot 2^{2r-1}-2^{\lfloor (3r-1)/2\rfloor}, & (t\equiv \pm2\bmod{2^r}).
\end{cases}
\end{align*}

When $p\geq3$ and $r\geq1$, we have 
\begin{align*}
&\#\left\{\gam\in \mathrm{SL}_2(\bZ_{p^r})\divset \tr{\gam}\equiv t\bmod{p^r}\right\}\\
&=\begin{cases}
p^{2r-1}(p-1), & ((T/p)=-1),\\
p^{2r-1}(p+1), & (\text{$(T/p)=1$ or $p^l\mmid T,2\mid l,(\frac{T}{p^l}/p)=1$}),\\
p^{2r}+p^{2r-1}-2p^{2r-l/2-1}, & (p^l\mmid T,2\mid l,(\frac{T}{p^l}/p)=-1),\\
p^{2r}+p^{2r-1}-p^{2r-(l+1)/2}-p^{2r-(l+3)/2}, & (p^l\mmid T, 2\nmid l),\\
p^{2r}+p^{2r-1}-p^{\lfloor (3r-1)/2\rfloor}, & (T\equiv 0\bmod{p^r}),
\end{cases}
\end{align*}
where $T:=t^2-4$.
\end{lem}

\begin{proof}
For simplicity, we prove only for $p\geq3$. 
Let $\gam:=\begin{pmatrix}a & b\\ c&d\end{pmatrix}\in\sz$. 
Since $\tr{\gam}=a+d\equiv t$ and $ad-bc\equiv 1$, 
$\#\left\{\gam\in \mathrm{SL}_2(\bZ_{p^r})\divset \tr{\gam}\equiv t\bmod{p^r}\right\}$ 
is the number of $a,b,c\in \bZ_{p^r}$ with
\begin{align}\label{abc}
bc\equiv T/4-(a-t/2)^2 \bmod{p^r}.
\end{align}

\noi{\bf (i) The case of $(T/p)=-1$.} 
In this case, there are no $a$ such that $bc=T/4-(a-t/2)^2$ is divided by $p$.
Then $c$ is uniquely determined by $c\equiv b^{-1}(T/4-(a-t/2)^2)$ for given $a\in\bZ_{p^r}$ and $b\in \bZ_{p^r}^*$. 
Thus the number of such $(a,b,c)$ is $p^r\cdot p^{r-1}(p-1)=p^{2r-1}(p-1)$.

\noi{\bf (ii) The case of $(T/p)=1$.}
The number of $a$ with $p\nmid bc$ is $p^{r-1}(p-2)$ and, for such $a$ and given $b\in \bZ_{p^r}^*$, 
$c$ is uniquely determined by $c\equiv b^{-1}(T/4-(a-t/2)^2)$. 
Then the number of such $(a,b,c)$ is $p^{2r-2}(p-1)(p-2)$.

The number of $a$ with $p^l\mmid bc$ $(1\leq l\leq r-1)$ is $2p^{r-l-1}(p-1)$. 
In this case, $b\equiv b_1p^{l_1}$ and $c\equiv c_1p^{l-l_1}$ for $0\leq l_1\leq l$, $b_1\in \bZ_{p^{r-l_1}}^*$ and $c_1\in\bZ_{p^{r-l+l_1}}^*$.
Then the number of $b_1,c_1$ for given $a$ and $l_1$ is $p^{r-l_1-1}(p-1)\cdot p^{l_1}=p^{r-1}(p-1)$.

The number of $a$ with $bc\equiv 0$ is $2$. 
In this case, $b\equiv b_1p^{l_1}$ and $c\equiv c_1^{l_2}$ for $0\leq l_1,l_2\leq r$, $l_1+l_2\geq r$, 
$b_1\in\bZ_{p^{r-l_1}}^*$ and $c_1\in\bZ_{p^{r-l_2}}^*$. 

Thus the total number of $(a,b,c)$ is 
\begin{align*}
p^{2r-2}(p-1)(p-2)&+\sum_{1\leq l\leq r-1}\sum_{0\leq l_1\leq l}2p^{2r-l-2}(p-1)^2\\
&+\sum_{0\leq l_1\leq r}\sum_{r-l_1\leq l_2\leq r}2\varphi(p^{r-l_1})\varphi(p^{r-l_2})=p^{2r}+p^{2r-1},
\end{align*}
where $\varphi(n):=\#\bZ_{n}^*$. 

\noi{\bf (iii) The case of $p\mid T$.}
Suppose that $p^l\mmid T$ $(1\leq l\leq r)$. 
If $p^s\mmid a$ for $0\leq 2s<l$, then $p^{2s}\mmid bc$. 
In this case, $b\equiv b_1p^{l_1}$ and $c\equiv c_1p^{2s-l_1}$ for $0\leq l_1\leq 2s$, $b_1\in \bZ_{p^{r-l_1}}^*$ and $c_1\in\bZ_{p^{r-2s+l_1}}^*$. 
Then the number of $(b,c)$ for given $a$ and $l_1$ is $p^{r-1}(p-1)$.

The number of $(b,c)$ for $p^s\mid a$ and $2s\geq l$ is different between the case
of $2\mid l$, $(\frac{T}{p^l}/p)=1$ and otherwise.
In the later case, there are no $a$ with $p^{l+1}\mid bc$. 
Then the number of $(a,b,c)$ is 
\begin{align*}
&\sum_{0\leq s< l/2} \sum_{0\leq l_1\leq 2s}p^{2r-s-2}(p-1)^2
+\sum_{0\leq l_1\leq l} p^{r-\lfloor (l+1)/2\rfloor}p^{r-1}(p-1)\\
=&\begin{cases}
p^{2r}+p^{2r-1}-2p^{2r-l/2-1}, & (p^l\mmid T,2\mid l,(\frac{T}{p^l}/p)=-1),\\
p^{2r}+p^{2r-1}-p^{2r-(l+1)/2}-p^{2r-(l+3)/2}, & (p^l\mmid T, 2\nmid l).
\end{cases}
\end{align*} 
In the former case, there are $a$ with $p^{l+1}\mid bc$. 
The number of $a$ with $p^{l+s}\mmid bc$ is $p^{r-l/2-1}(p-2)$ when $s=0$, 
$2p^{r-l/2-s-1}(p-1)$ when $1\leq s\leq r-l-1$ and $2$ when $s=r-l$. 
Since the number of $(b,c)$ are given similar to the case (ii), 
we can calculate the number of $(a,b,c)$ as follows.
\begin{align*}
&\sum_{0\leq s< l/2} \sum_{0\leq l_1\leq 2s}p^{2r-s-2}(p-1)^2
\sum_{0\leq l_1\leq l}p^{2r-l/2-2}(p-1)(p-2)\\
&+\sum_{1\leq s\leq r-l-1}\sum_{0\leq l_1\leq l+s}2p^{2r-l/2-s-2}(p-1)^2
+\sum_{0\leq l_1\leq r}\sum_{r-l_1\leq l_2\leq r}2\varphi(p^{r-l_1})\varphi(p^{r-l_2})\\
=&p^{2r}+p^{2r-1}.
\end{align*}

The number of $(a,b,c)$ in the case of $p^r\mid T$ is computed similarly.
\end{proof}

\subsection{Proof of Theorem \ref{thm}} 

According to Lemma \ref{Idq}, we see that $I_{\Gam}\in\cD^q$ for any $q\geq1$ and 
$I_{\Gam}$ is approximated by $\beta_{\Gam,P,M}$ as $M,P\tinf$ with $M>4\log\log{P}$. 
Then Proposition \ref{lemproduct2} holds for $f_i(t)=I_{\Gam}(t+r_i)$ 
and $\{f_{ij}(t)\}_j=\{\beta_{\Gam,P,M}(t+r_i)\}_{P,M}$ 
with the periods $\{N_j\}_j=\{n^2B_{P,M}\}_{P,M}$.
Thus we have 
\begin{align}
c_{\Gam}^{(k)}({\bf r},j):=N_j^{k-1}\sum_{m\in \bZ_{N_j}}F_{1j}(m;N_j)\cdots F_{kj}(m;N_j)\to c_{\Gam}^{(k)}({\bf r}) \as j\tinf.
\end{align}

We now compare $c_{\Gam}^{(k)}({\bf r},j)$ and 
\begin{align*}
\tilde{c}_{\Gam}^{(k)}({\bf r},j):=N_j^{k-1}\sum_{m\in \bZ_{N_j}}F_{1}(m;N_j)\cdots F_{k}(m;N_j),
\end{align*}
where $F_i(m;N_j):=F_{\Gam}(m+r_i;N_j)$ is given in Theorem \ref{thm}.  
The difference is bounded by 
\begin{align}
\left|c_{\Gam}^{(k)}({\bf r},j)-\tilde{c}_{\Gam}^{(k)}({\bf r},j)\right|
\leq & N_j^{k-1}\sum_{m\in \bZ_{N_j}}\sum_{1\leq l\leq k} F_{1j}(m;N_j)\cdots F_{l-1,j}(m;N_j)\notag\\
&\times \left|F_{lj}(m;N_j)-F_l(m;N_j) \right|F_{l+1}(m;N_j)\cdots F_{k}(m;N_j).\label{diff}
\end{align} 
By the definition of $F_{ij}$ and $\beta_{\Gam,P,M}$, we have
\begin{align}\label{Fij}
N_jF_{ij}(m;N_j)&=\beta_{\Gam,P,M}(t+r_j)\nt\\
&\leq [\sz:\Gam]\Bigg(\prod_{p_1\leq P}\sum_{0\leq l\leq M}p_1^{-l}\Bigg)\sum_{p_2\leq P}(1-p_2^{-1})^{-1}\ll (\log{P})^2.
\end{align}
For $F_i$, since 
\begin{align}\label{multip}
F_i(m;n_1n_2)=F_i(m;n_1)F_i(m;n_2)
\end{align}
for relatively prime $n_1$ and $n_2$ and $\Gam\hGam(N)=\sz$ except a finite number of $N$,
we see that 
\begin{align*}
N_jF_i(m;N_j)\ll \prod_{p\leq P} 
\frac{p^{2M+1}\#\{\gam\in {\rm SL}_2(\bZ_{p^{2M+1}})\divset\tr\equiv m+r_i\bmod{p^{2M+1}}\}}
{\#{\rm SL}_2(\bZ_{p^{2M+1}})}.
\end{align*}
Due to Lemma \ref{trace} and $\#\slt(\bZ_{p^r})=p^{3r-2}(p^2-1)$, we have  
\begin{align}\label{Fi}
N_jF_i(m;N_j)\ll \prod_{p\leq P}\frac{p^{2M+1}p^{4M+1}(p+1)}{p^{6M+1}(p^2-1)}=\prod_{p\leq P}(1-p^{-1})^{-1}\ll \log{P}.
\end{align}
Applying \eqref{Fij}, \eqref{Fi} and Lemma \ref{Idq} into \eqref{diff}, 
we obtain 
\begin{align}
\left|c_{\Gam}^{(k)}({\bf r},j)-\tilde{c}_{\Gam}^{(k)}({\bf r},j)\right|
\ll &(\log{P})^{2k-2}\sum_{1\leq l\leq k}\sum_{m\in \bZ_{N_j}}\left|F_{lj}(m;N_j)-F_l(m;N_j) \right|\notag\\
\leq &(\log{P})^{2k-2}\sum_{1\leq l\leq k}||f_{lj}-f_l||_1\notag\\
\ll &(\log{P})^{2k-2}(P^{-\epsilon}+2^{-M}(\log{P})^2) \to 0 
\end{align}
as $P,M\tinf$ with $M\gg \log{P}$, and then $c_{\Gam}^{(k)}({\bf r})=\lim_{j\tinf}\tilde{c}_{\Gam}^{(k)}({\bf r},j)$.
Thus the expression of $c_{\Gam}^{(k)}({\bf r})$ as a product over $p$ follows from the multiplicative property \eqref{multip} of $F_i$. 
 \qed

\section{Examples}

In this section, we calculate $c_{\Gam}^{(k)}(0)$ for $\Gam=\sz,\Gam_0(n),\hGam(n)$.

\subsection{The case of $\Gam=\sz$}
In this case, it is clear that $\Gam\Gam(p^l)=\sz$ and $\#{\rm SL}_2(\bZ_{p^r})=p^{3r-2}(p^2-1)$. 
We now propose the following lemma. 

\begin{lem}\label{tracesz}
Let $p^r$ be a power of prime and $1\leq l\leq r-1$. 
Then we have
\begin{align*}
\#\{t\in\bZ_{p^r} \divset (T/p)=1\}=&p^{r-1}(p-3)/2,\\
\#\{t\in\bZ_{p^r} \divset (T/p)=-1\}=&p^{r-1}(p-1)/2,\\
\#\left\{t\in\bZ_{p^r} \divset p^l\mmid T, \left(Tp^{-l}/p\right)=1\right\}=&p^{r-1}(p-1),
\end{align*}
\begin{align*}
\#\left\{t\in\bZ_{p^r} \divset p^l\mmid T, \left(Tp^{-l}/p\right)=-1\right\}=&p^{r-1}(p-1),\\
\#\{t\in\bZ_{p^r} \divset T\equiv 0\bmod{p^r}\}=&2.
\end{align*}
\end{lem}

\begin{proof}
Let $\alpha\in \bZ_{p^r}^*-\{\pm1\}$ and $t(\alpha):=\alpha+\alpha^{-1}\bmod{p}$. 
Then $t(\alpha)^2-4=(\alpha-\alpha^{-1})^2$ is a quadratic residue of $p$. 
Conversely, suppose that $t\in \bZ_p$ satisfies $(T/p)=1$ 
and $\alpha(t):=(t+\sqrt{t^2-4})/2$. 
Then we see that $\alpha(t)\in (\bZ_p)^*-\{\pm1\}$ and $\alpha(t)+\alpha(t)^{-1}=t$. 
Thus we have 
\begin{align*}
\#\{t\in\bZ_p \divset (T/p)=1\}=\#\{\alpha\in \bZ_p^*-\{\pm1\}\}/2=(p-3)/2
\end{align*}
and the result for $(T/p)=1$ follows immediately. Since 
\begin{align*}
\#\{t\in\bZ_{p^r} \divset p\mid T\}=\#\{t\in\bZ_{p^r} \divset t\equiv \pm 2\bmod{p} \}=2p^{r-1},
\end{align*}
we also get the result for $(T/p)=-1$.

If $p^l\mmid T$ then $t=\pm2+t_1p^l$ for $t_1\in \bZ_{p^{r-l}}^*$ 
and $T/p^l=t_1^2p\pm4t_1\equiv \pm4t_1\bmod{p}$. 
Thus the results for $p^l\mmid T$ is obtained easily.
\end{proof}

Combining Lemma \ref{trace} and \ref{tracesz}, we get 
\begin{align*}
\lim_{r\tinf}2^{r(k-1)}\sum_{m\in\bZ_{2^r}}F_{\sz}(m;2^r)
=&\frac{1}{2}\left(\frac{2}{3}\right)^k+\frac{1}{4}+2^{k-6}+\frac{3}{64}\left(\frac{5}{3}\right)^k
+\sum_{\begin{subarray}{c}3\leq l\leq r-1 \\ \text{odd}\end{subarray}}2^{k-l}(1-2^{-(l+1)/2})^k\\
&+\sum_{\begin{subarray}{c}6\leq l\leq r-1 \\ \text{even}\end{subarray}}2^{k-l-2}(1-3(1-3^{-1}\cdot 2^{-l/2+1})^k),
\end{align*}
and, for $p\geq3$, 
\begin{align*}
&\lim_{r\tinf}p^{r(k-1)}\sum_{m\in\bZ_{p^r}}F_{\sz}(m;p^r)\\
=&\frac{1}{2}(1-3p^{-1})(1+p^{-1})^{-k}+\frac{1}{2}(1-p^{-1})^{-k+1}+p^{-2}(1+p^{-1})(1-p^{-1})^{-k}\\
&+\sum_{l\geq1}2p^{-2l+1}(1-p^{-1})^{-k+1}(1-p^{-l})^k
+\sum_{l\geq1}p^{-2l}(1-p^{-1})(1-p^{-2})^{-k}(1+p^{-1}-2p^{-l-1})^k.
\end{align*}
According to Theorem \ref{thm}, we see that $c_{\sz}^{(k)}(0)$ is the product of these values over primes $p$. 
Especially, for $k=2,3$, we have 
\begin{align*}
c_{\sz}^{(2)}(0)=&\frac{1015}{864}\prod_{p\geq3}\frac{p^2(p^3+p^2-p-3)}{(p-1)^2(p+1)^3},\\
c_{\sz}^{(3)}(0)=&\frac{682495}{428544}\prod_{p\geq3}\frac{p^8+p^7+p^6-5p^5-5p^3-5p^2-p-1}{(p-1)^2(p+1)^2(p^4+p^3+p^2+p+1)}.
\end{align*}
Note that $c_{\sz}^{(2)}(0)$ coincides with the coefficient given in \cite{BLS,Pe}.

\subsection{The case of $\Gam=\Gam_0(n)$}

Let $n$ be an odd integer and 
$\Gam_0(n)$ be a congruence subgroup of $\sz$ consisting of the elements $\gam$ with $n\mid \gam_{21}$.
We see that $\Gam\Gam(p^r)=\Gam_0(p^{\min(r,\ord_p{n})})$ for $p\mid n$ and $\Gam\Gam(p^r)=\sz$ for $p\nmid n$. 
Since $[\sz:\Gam_0(p^N)]=p^{N-1}(p+1)$, we have $\#\Gam_0(p^N)/\Gam(p^r)=p^{3r-N-1}(p-1)$ for $N\leq r$.
The value $\#\left\{\gam\in \Gam_0(p^N)/\Gam(p^r)\divset \tr{\gam}\equiv t\bmod{p^r}\right\}$ is given as follows.
\begin{lem}\label{gam0}
\begin{align*}
&\#\left\{\gam\in \Gam_0(p^N)/\Gam(p^r)\divset \tr{\gam}\equiv t\bmod{p^r}\right\}\\
&=\begin{cases}
2p^{2r+l/2-N-1}, & (p^l\mmid T,0\leq l\leq N, 2\mid l,(\frac{T}{p^l}/p)=1),\\
p^{2r-\lfloor (N+1)/2\rfloor}+p^{2r-\lfloor N/2\rfloor-1}, & (p^l\mmid T,l\geq N, 2\mid l,(\frac{T}{p^l}/p)=1),\\
p^{2r-\lfloor (N+1)/2\rfloor}+p^{2r-\lfloor N/2\rfloor-1}-2p^{2r-l/2-1}, &  (p^l\mmid T,l\geq N, 2\mid l,(\frac{T}{p^l}/p)=-1),\\
p^{2r-\lfloor (N+1)/2\rfloor}+p^{2r-\lfloor N/2\rfloor-1}-p^{2r-(l+1)/2}-p^{2r-(l+3)/2}, &
 (p^l\mmid T,l\geq N, 2\nmid l),\\
p^{2r-\lfloor (N+1)/2\rfloor}+p^{2r-\lfloor N/2\rfloor-1}-p^{\lfloor (3r-1)/2\rfloor}, & (T\equiv 0),\\
0, & (\text{otherwise}).
\end{cases}
\end{align*}
\end{lem}

\begin{proof}
The element $\gam:=\begin{pmatrix}a & b\\ c&d\end{pmatrix}\in\Gam_0(p^N)$ satisfies $p^N\mid c$. 
Then, picking up such $\gam$'s, we can calculate $\#\left\{\gam\in \Gam_0(p^N)/\Gam(p^r)\divset \tr{\gam}\equiv t\bmod{p^r}\right\}$ 
similar to the case of $\Gam=\sz$.
\end{proof}

Computing $c_{\Gam_0(n)}^{(k)}(0)$ similar to the case of $\Gam=\sz$, we get
\begin{align*}
c_{\Gam_0(n)}^{(2)}(0)=&\frac{1015}{864}\prod_{p\geq3,p\nmid n}\frac{p^2(p^3+p^2-p-3)}{(p-1)^2(p+1)^3}\prod_{p^N\mmid n}\frac{2p(Np^2-p-N)}{(p-1)^2(p+1)},
\end{align*}
\begin{align*}
c_{\Gam_0(n)}^{(3)}(0)=&\frac{682495}{428544}\prod_{p\geq3,p\nmid n}\frac{p^8+p^7+p^6-5p^5-5p^3-5p^2-p-1}{(p-1)^2(p+1)^2(p^4+p^3+p^2+p+1)},\\
&\times\prod_{p^N\mmid n}\frac{p^2(p^{\lfloor (N-1)/2\rfloor} h_1(p)-2h_2(p))}{(p-1)^3(p+1)(p^4+p^3+p^2+p+1)},
\end{align*}
where 
\begin{align*}
h_1(p):=&\begin{cases}(p+1)(p^6+2p^5+5p^4+2p^3+5p^2+2p+1), & (2\mid N) \\ 2(4p^6+9p^5+14p^4+12p^3+14p^2+9p+4), & (2\nmid N), \end{cases}\\
h_2(p):=&\begin{cases}(p+1)^2(p^4+3p^3+p^2+3p+1), & (2\mid N),\\ 2(p+1)(p+3)(p^4+p^3+p^2+p+1), & (2\nmid N). \end{cases}
\end{align*}
Note that $c_{\Gam_0(n)}^{(2)}(0)$ for square free $n$ coincides with the coefficient given in \cite{Lu}.

\subsection{The case of $\Gam=\hGam(n)$}
Let $n$ be an odd integer. 
When $\Gam=\hGam(n)$, we see that $\Gam\Gam(p^r)=\hGam(p^{\min(r,\ord_p{n})})$ for $p\mid n$ 
and $\Gam\Gam(p^r)=\sz$ for $p\nmid n$. 
Since $[\sz:\hGam(p^N)]=p^{3N-2}(p^2-1)/2$, we have $\#\Gam\Gam(p^r)=2p^{3(r-N)}$ for $N\geq r$. 
The value $\#\left\{\gam\in \Gam_0(p^N)/\Gam(p^r)\divset \tr{\gam}\equiv t\bmod{p^r}\right\}$ is given as follows.
\begin{lem}
\begin{align*}
&\#\left\{\gam\in \hGam(p^N)/\Gam(p^r)\divset \tr{\gam}\equiv t\bmod{p^r}\right\}\\
\end{align*}
\begin{align*}
&=\begin{cases}
p^{2r-N-1}(p-1), &((T/p)=-1),\\
p^{2r-N-1}(p+1), & (\text{$(T/p)=1$ or $p^l\mmid T,2\mid l, (\frac{T}{p^l}/p)=1)$}),\\
p^{2r-N}+p^{2r-N-1}-2p^{2r-l/2-1}, & (p^l\mmid T,2\mid l, (\frac{T}{p^l}/p)=-1),\\
p^{2r-N}+p^{2r-N-1}-p^{2r-(l+1)/2}-p^{2r-(l+3)/2}, & (p^l\mmid T,2\nmid l),\\
p^{2r-N}+p^{2r-N-1}-p^{\lfloor (3r-1)/2\rfloor}, & (T\equiv 0),\\
0, & (t\not\equiv \pm2 \bmod{p^{2N}}),
\end{cases}
\end{align*}
where $T:=(t\pm2)/p^{2N}$ for $t\equiv \mp2\bmod{p^{2N}}$.
\end{lem}
\begin{proof}
The element $\gam:=\begin{pmatrix}a & b\\ c&d\end{pmatrix}\in\hGam(p^N)$ satisfies $p^{N}\mid b,p^N\mid c, a\equiv d\equiv \pm1\bmod{p^N}$. 
Then, picking up such $\gam$'s, we can get the result similar to the case of $\Gam=\sz$.
\end{proof}

Computing $c_{\hGam(n)}^{(k)}(0)$ similar to the case of $\Gam=\sz$, we get
\begin{align*}
c_{\hGam(n)}^{(2)}(0)=&\frac{1015}{864}\prod_{p\geq3,p\nmid n}\frac{p^2(p^3+p^2-p-3)}{(p-1)^2(p+1)^3}
\prod_{p^N\mmid n}\frac{p^{2N-1}(p^2+p+1)}{2(p+1)},\\
\end{align*}
\begin{align*}
c_{\hGam(n)}^{(3)}(0)=&\frac{682495}{428544}\prod_{p\geq3,p\nmid n}\frac{p^8+p^7+p^6-5p^5-5p^3-5p^2-p-1}{(p-1)^2(p+1)^2(p^4+p^3+p^2+p+1)}\\
&\times \prod_{p^N\mmid n}\frac{p^{4N-2}(p^6+p^5+4p^4+p^3+4p^2+p+1)}{4(p^4+p^3+p^2+p+1)}.
\end{align*}

\end{document}